\documentclass[11pt]{scrartcl}
\usepackage{graphicx} % Required for inserting images
\usepackage{fullpage}
\usepackage{amsmath, amsthm, amssymb}
\usepackage{thmtools}
\usepackage[shortlabels]{enumitem}
\usepackage{hyperref}
\usepackage{enumitem}

\usepackage{booktabs, multirow}

\newtheorem{theorem}{Theorem}[section]% Numbering is impacted by [section]; could do [subsection] also.
\newtheorem{lemma}[theorem]{Lemma} % The [thm] argument says to number Lemma in sequence with Theorem.

\newtheorem{question}[theorem]{Question}
\newtheorem{conjecture}[theorem]{Conjecture}

\theoremstyle{definition}
\newtheorem{definition}[equation]{Definition} % These definitions are also numbered in sequence with Theorem.

\theoremstyle{remark}
\newtheorem{claim}{Claim}

\usepackage{mathtools}
\usepackage{xcolor}

\title{Nearly all known Euclidean Ramsey sets are subsoluble}
\author{Natalie Behague}
\newcommand{\kriz}{K\v{r}\'{i}\v{z}}

\begin{document}

\maketitle

\begin{abstract}
A finite set $X$ in a Euclidean space $\mathbb{R}^d$ is called \emph{Ramsey} if for every $k$ there exists an integer $n$ such that whenever $\mathbb{R}^n$ is coloured with $k$ colours, there is a monochromatic copy of $X$. Graham conjectured that all spherical sets are Ramsey, but progress on this conjecture has been slow. A key result of K\v{r}\'{i}\v{z} is that  all sets that embed in sets that are acted on transitively by a soluble group are Ramsey. We show that for nearly all known examples of Ramsey sets the converse is true, with only two possible exceptions.
\end{abstract}
\section{Introduction}
A finite set $X$ in a Euclidean space $\mathbb{R}^d$ is called \emph{Ramsey} if for every $k$ there exists an integer $n$ such that whenever $\mathbb{R}^n$ is coloured with $k$ colours, there is a monochromatic isometric copy of $X$. An early paper \cite{originaleuclidram}  in the study of Euclidean Ramsey theory proved that every Ramsey set is \emph{spherical}, that is, it lies on the surface of some $d$-dimensional sphere. Graham conjectured that the converse holds.
\begin{conjecture}[\cite{graham}] \label{conj:Graham}
    Every spherical set is Ramsey.
\end{conjecture}

This conjecture is considered the most important in Euclidean Ramsey theory and yet progress has been slow. It is known that rectangular parallelepipeds (a.k.a.~bricks)~\cite{originaleuclidram}, triangles~\cite{triangles}, (non-degenerate) simplices~\cite{simplices}, isosceles trapezia (a.k.a.~isosceles trapezoids)\footnote{an isosceles trapezium, also called an isosceles trapezoid, is a quadrilateral with one pair of parallel sides and the other pair of sides equal in length.}~\cite{trapezia} and all regular polytopes~\cite{kriz, cantwell} are Ramsey.  
A key breakthrough, which was critical in proving that all regular polytopes are Ramsey, is the following result of \kriz.

\begin{theorem}[\cite{kriz}]\label{thm:Kriz}
    Let $X \subset \mathbb{R}^d$ be a finite transitive set. If there is a soluble group $G$ that acts transitively on $X$ then $X$ is Ramsey.
\end{theorem}

 (\kriz~actually also proved the stronger result that if there is a soluble group $G$ that acts on $X$ with at most two orbits, then $X$ is Ramsey -- we will return to this point at the end of the paper.) 
Inspired by Theorem~\ref{thm:Kriz}, we make the following definitions.

\begin{definition} We say a finite set $X \subset \mathbb{R}^d$ is
\begin{enumerate}
    \item  \emph{soluble} if there is a soluble group that acts transitively on $X$, and
    \item  \emph{subsoluble} if there is a soluble set $Y$ containing a copy of $X$.
\end{enumerate}
Note that $Y$ can lie in a higher-dimensional space than $X$.
\end{definition}

By Theorem~\ref{thm:Kriz}, we know that all soluble sets are Ramsey, and as a consequence all subsoluble sets are Ramsey. 
In this paper we will explore the converse, and ask whether there are any Ramsey sets that are not subsoluble. For example, it is easy to observe that rectangular parallelepipeds are soluble (see Section~\ref{sec:directprod}).  Karamanlis~\cite{Karamanlis} generalised the proof of Frankl and R\"{o}dl that all simplices are Ramsey to obtain the following result.
\begin{theorem}[\cite{Karamanlis}]
    Every (non-degenerate) simplex embeds into a regular polygonal torus, and is therefore subsoluble.
\end{theorem}

%Despite this progress, the conjecture remains wide open in general, with cyclic quadrilaterals a significant open case.
Inspired by the known proofs in Euclidean Ramsey theory, Leader, Russell and Walters \cite{blocksets} proposed a `rival' conjecture to Graham's.
\begin{conjecture}[\cite{blocksets}]\label{conj:rival}
    A set $X \subset \mathbb{R}^d$ is Ramsey if and only if it is (isometric to) a subset of a finite transitive set.
\end{conjecture}
Leader, Russell and Walters \cite{blocksets} showed that there exist cyclic quadrilaterals that are spherical but not a subset of any finite transitive set, and therefore the two conjectures are distinct.

To make progress on Conjecture~\ref{conj:rival}, Leader, Russell and Walters~\cite{blocksets} gave a slightly stronger but purely combinatorial reframing called the Block Sets Conjecture which has a similar statement to the famous Hales-Jewett theorem. They proved the first non-trivial case of the Block Sets conjecture and noted that as a consequence, for any $\alpha, \beta,\gamma \in \mathbb{R}$ and $i,j \in \mathbb{N}$, the set $X \subset \mathbb{R}^{i+j+1}$ consisting of permutations of the coordinates $(\underbrace{\alpha,\ldots,\alpha}_{i},\underbrace{\beta,\ldots,\beta}_{j},\gamma)$ is Ramsey. They observe that if $\alpha, \beta, \gamma$ are algebraically independent, these sets are not themselves soluble and speculate that they may therefore be genuinely new Ramsey sets that cannot be found as a consequence of Theorem~\ref{thm:Kriz}.

Perhaps surprisingly, we show that these sets, while not themselves soluble, do embed into soluble sets. Our methods actually extend to almost every known Ramsey set, with at most two exceptions.

\begin{theorem}\label{thm:main}
    The following sets are subsoluble:
    \begin{enumerate}
        \item \label{item:blocksets} The set of permutations of $(\underbrace{\alpha,\ldots,\alpha}_{i},\underbrace{\beta,\ldots,\beta}_{j},\gamma)$
        for any $\alpha,\beta,\gamma \in \mathbb{R}$ and $i,j \ge 0$,
        \item all regular polytopes except for possibly the 120-cell and 600-cell, and \label{item:polytopes}
 %       \item (non-degenerate) simplices, and
         \item isosceles trapezia.
    \end{enumerate}
\end{theorem}

In case \eqref{item:blocksets} we can also prove a more general result. A proof of the next two open cases of the Block Sets Conjecture would show that the set of permutations of the coordinates $(\underbrace{\alpha,\ldots,\alpha}_{i},\underbrace{\beta,\ldots,\beta}_{j},\gamma,\gamma)$ and the set of permutations of $(\underbrace{\alpha,\ldots,\alpha}_{i},\underbrace{\beta,\ldots,\beta}_{j},\gamma,\delta)$ are both Ramsey for all $\alpha, \beta, \gamma,\delta \in \mathbb{R}$, $i,j \ge 0$. These sets were not previously known to be Ramsey, but we are able to prove that they are in fact subsoluble.

\begin{theorem}\label{thm:blocksets subsoluble}
      Let $\alpha,\beta,\gamma,\delta \in \mathbb{R}$ (not necessarily distinct), and let $i,j \ge 0$. The following sets are subsoluble and therefore Ramsey:
    \begin{align*}\{\text{permutations of }(\underbrace{\alpha,\ldots,\alpha}_{i},\underbrace{\beta,\ldots,\beta}_{j})\},~~
    &\{\text{permutations of }(\underbrace{\alpha,\ldots,\alpha}_{i},\underbrace{\beta,\ldots,\beta}_{j},\gamma)\},\\
    \{\text{permutations of }\underbrace{\alpha,\ldots,\alpha}_{i},\underbrace{\beta,\ldots,\beta}_{j},\gamma,\gamma)\}, ~~
    &\{\text{permutations of }(\underbrace{\alpha,\ldots,\alpha}_{i},\underbrace{\beta,\ldots,\beta}_{j},\gamma,\delta)\}.\end{align*}
\end{theorem}
Unfortunately these results do not go backwards to prove the corresponding cases of the Block Sets Conjecture.

\subsection{Paper Outline}

We begin with some basic preliminaries in Section~\ref{sec:prelim}. We then prove Theorem~\ref{thm:blocksets subsoluble} in Section~\ref{sec:blocksets}.

Section~\ref{sec:polytopes} contains the proof that nearly all regular polytopes are subsoluble, bar two exceptions. It is easy to check that regular $k$-gons and $d$-dimensional regular simplices, cubes and orthoplexes are soluble (acted on by $C_k$, $C_d$, $C_2^d$ and $C_2\times C_d$ respectively). The non-trivial cases are the icosahedron and dodecahedron  with the dodecahedron being the most difficult.

% The proof that all simplices are subsoluble follows from an analysis of the proof of Frankl and R\"odl~\cite{simplices} that all  simplices are Ramsey. We include this analysis for completeness in Section~\ref{sec:simplices}. 
We give the construction of a soluble set containing an arbitrary isosceles trapezium in Section~\ref{sec:trapezia}.
Finally, we finish with several open questions in Section~\ref{sec:openqs}.

\section{Preliminaries}\label{sec:prelim}
\subsection{Basic Group Properties}

Let $C_n$ be the cyclic group of order $n$. For general $n$ we will denote the elements of this group by $g^i, 0\le i \le {n-1}$, but for ease we will usually let the elements of $C_2$ be $\{1,-1\}$.

For a prime $p$, let $\mathbb{Z}_p$ be the field of order $p$. The affine group $AGL(1,p)$ contains all maps of the form
\begin{align*}
    \phi: \mathbb{Z}_p &\rightarrow \mathbb{Z}_p \\
    a &\mapsto ta + s
\end{align*} 
where $t \in \mathbb{Z}_p\setminus\{0\}$ and $s \in \mathbb{Z}_p$. Note that as a group $AGL(1,p)$ is isomorphic to the semi-direct product $C_{p-1} \rtimes C_p$ and since the semi-direct product of soluble groups is soluble, $AGL(1,p)$ is soluble.

\subsection{Direct products are subsoluble}\label{sec:directprod}

The following simple lemma will be useful later.
\begin{lemma}\label{lem:direct product}
    If $X,Y$ are subsoluble then so is the direct product $X \times Y$.
\end{lemma}
\begin{proof}
As $X$ is subsoluble, there exists a finite set $X'$ containing $X$ with soluble transitive group action $G$. Similarly, there exists a finite set $Y'$ containing $Y$ with soluble transitive group action $H$. Clearly $X' \times Y'$ is a finite set containing $X \times Y$ and $G \times H$ is soluble. Moreover, $G \times H$ acts transitively on $X' \times Y'$ via the natural action $(g,h)(x,y) = (g(x), h(y))$ for $g\in G$, $h \in H$, $x \in X'$ and $y \in Y'$.
\end{proof}

Note that as an immediate corollary we have that rectangular parallelepipeds are subsoluble.

\section{Known Ramsey sets arising from block sets are subsoluble}\label{sec:blocksets}

In this section we prove Theorem~\ref{thm:blocksets subsoluble}. The following more general lemma is the key result for the proof. 

\begin{lemma}\label{lem:blockset subsoluble}
    Let $\alpha,\beta,\gamma \in \mathbb{R}$ (not necessarily distinct), let $\ell \ge 0$ and let $X\subset \mathbb{R}^{\ell+2}$  be the set of all permutations of coordinates \[(\underbrace{\pm\alpha,\ldots,\pm\alpha}_{\ell},\pm \beta, \pm \gamma),\] where $\pm x$ means take both the option with $x$ and with $-x$. Then $X$ is subsoluble.
\end{lemma}

Before we prove the lemma we will show how to use it to obtain the theorem.
\begin{proof}[Proof of Theorem~\ref{thm:blocksets subsoluble}]
     Let $\alpha,\beta,\gamma,\delta \in \mathbb{R}$ (not necessarily distinct), and let $i,j,k,\ell \ge 0$. 
      Let $X(i,j,k,\ell) \subset \mathbb{R}^{i+j+k+\ell}$ be the set containing permutations of the coordinates \[(\underbrace{\alpha,\ldots,\alpha}_{i},\underbrace{\beta,\ldots,\beta}_{j},\underbrace{\gamma,\ldots,\gamma}_{k},\underbrace{\delta,\ldots,\delta}_{\ell}).\] 
If $k = \ell = 1$, we will show that there is a copy of $X(i,j,1,1)$ in a set described in the statement of Lemma~\ref{lem:blockset subsoluble} with appropriate choices for the variables. 
Consider the set $Y$ of all permutations of the coordinates
\[\left(\underbrace{\pm \frac{\alpha - \beta}{2},\ldots ,\pm \frac{\alpha - \beta}{2}}_{i+j}, \pm\left(\gamma - \frac{\alpha + \beta}{2}\right), \pm\left(\delta - \frac{\alpha + \beta}{2}\right)\right).\]
Consider the subset $X \subset Y$ containing all permutations of the coordinates
\[\left(\underbrace{\frac{\alpha - \beta}{2},\ldots ,\frac{\alpha - \beta}{2}}_{i}, \underbrace{ \frac{\beta - \alpha}{2},\ldots , \frac{\beta - \alpha}{2}}_{j}, \gamma - \frac{\alpha + \beta}{2}, \delta - \frac{\alpha + \beta}{2}\right).\]
If we translate this set by adding $\left(\frac{\alpha + \beta}{2},\frac{\alpha + \beta}{2},\ldots,\frac{\alpha + \beta}{2}\right)$ to all elements,
we get all permutations of $(\underbrace{\alpha,\ldots,\alpha}_{i},\underbrace{\beta,\ldots,\beta}_{j},\gamma,\delta)$, which is exactly $X(i,j,1,1)$. Thus there is a copy of $X(i,j,1,1)$ in $Y$ and so by Lemma~\ref{lem:blockset subsoluble}, $X(i,j,1,1)$ is subsoluble.

The case where $k=0$ and $\ell=2$ or vice versa follows immediately from the previous case by setting $\gamma = \delta$.

Let $i' \ge i,j'\ge j, k'\ge k$ and $\ell'\ge \ell$. It is not hard to see that within $X(i',j',k',\ell')$ there is a copy of $X(i,j,k,\ell)$: take all the elements of $X(i',j',k',\ell')$ with initial coordinates $\underbrace{\alpha,\ldots,\alpha}_{i'-i},\underbrace{\beta,\ldots,\beta}_{j'-j},\underbrace{\gamma,\ldots,\gamma}_{k'-k},\underbrace{\delta,\ldots,\delta}_{\ell'-\ell}$ in exactly that order. Therefore if $X(i',j',k',\ell')$ is subsoluble then so is $X(i,j,k,\ell)$. This completes the proof.
\end{proof}

We will now prove the lemma.
\begin{proof}[Proof of Lemma~\ref{lem:blockset subsoluble}]
    Let $p$ be a prime larger than $2+\ell$, and let $X'$ be all permutations of the coordinates $(\underbrace{\pm\alpha,\ldots,\pm\alpha}_{p-2},\pm \beta, \pm \gamma)$. Clearly $X'$ contains a copy of $X$: take all elements of $X'$ where the initial $p - \ell - 2$ entries are $\alpha$.
    Let $G$ be the wreath product $C_2 \wr AGL(1,p)$; that is, $G = C_2^p \rtimes AGL(1,p)$, where $(i_1,\ldots,i_p;\phi)\cdot (j_1,\ldots,j_p;\psi) = (i_1j_{\phi^{-1}(1)},\ldots,i_pj_{\phi^{-1}(p)};\phi\psi)$.  Note that $G$ is a soluble group as $AGL(1,p) $ and $ C_2$ are both soluble.
    
    Let $\{1,-1\}$ be the elements of $C_2$ and define an action of $G$ on $X'$ where for $(i_1,\ldots,i_p;\phi) \in  G'$ we have
    \[(i_1,\ldots,i_p;\phi):(x_1,x_2,\ldots,x_p) \mapsto (i_1x_{\phi^{-1}(1)}, i_2x_{\phi^{-1}(2)},\ldots, i_px_{\phi^{-1}(p)}).  \]
    Clearly if $(x_1,x_2,\ldots,x_p) \in X'$ then so is $(i_1x_{\phi(1)}, i_2x_{\phi(2)},\ldots, i_px_{\phi(p)})$, so this is well-defined. Moreover, \begin{align*}
    (i_1,\ldots,i_p;\phi)(j_1,\ldots,j_p;\psi)(x_1,x_2,\ldots,x_p) &= (i_1,\ldots,i_p;\phi) (j_1x_{\psi^{-1}(1)}, j_2x_{\psi^{-1}(2)},\ldots, j_px_{\psi^{-1}(p)}) \\ &= 
    (i_1j_{\phi^{-1}(1)}x_{\psi^{-1}\phi^{-1}(1)},\ldots,i_pj_{\phi^{-1}(p)}x_{\psi^{-1}\phi^{-1}(p)})
    \\ &= (i_1j_{\phi^{-1}(1)},\ldots,i_pj_{\phi^{-1}(p)};\phi\psi)(x_1,x_2,\ldots,x_p) 
    \end{align*}
    and so this is a group action.
    
    All that remains is to show that the action is transitive.
    We will show that $(\beta, \gamma,\underbrace{\alpha,\ldots,\alpha}_{p-2})$ can be transformed under the action of $G$ to any $(x_1,x_2,\ldots,x_p) \in X'$ and vice versa.
    Let $t \ne s$ be such that $x_t \in \{\beta,-\beta\}$ and $x_s = \{\gamma,-\gamma\}$. Let $\phi:\mathbb{Z}_p \rightarrow \mathbb{Z}_p$ send $a \mapsto (s-t)a + (2t-s)$ so that $\phi(1) = t$ and $\phi(2) = s$. Then for $1 \le j \le p$, pick $i_j \in \{0,1\}$ such that $i_t\beta = x_t$, $i_s\gamma = x_s$ and for $j \ne s,t$, $i_j\alpha = x_j$.
    Taking $h = (i_1,i_2,\ldots,i_p;\phi)$ gives
    \begin{align*}
    h((\beta, \gamma,\underbrace{\alpha,\ldots,\alpha}_{p-2})) &= (i_1\alpha,i_2\alpha,\ldots,i_t\beta, \ldots, i_s\gamma, \ldots, i_p\alpha) \\ &= (x_1,x_2,\ldots,x_p).
    \end{align*}
    Clearly $h^{-1}(x_1,x_2,\ldots,x_p) = (\beta, \gamma,\underbrace{\alpha,\ldots,\alpha}_{p-2})$. Thus the action is transitive, as required.
\end{proof}

\section{Nearly all regular polytopes are subsoluble}\label{sec:polytopes}
In this section we prove that all regular polytopes are subsoluble except for possibly two.
\begin{theorem}\label{thm:regular polytopes subsoluble}
    The vertex sets of all regular polytopes are subsoluble, except for possibly the $120$-cell and the $600$-cell. 
    In particular, the following regular polytopes have soluble symmetry groups:
    \begin{itemize}[topsep = 3pt, itemsep=0pt]
        \item In $2$ dimensions: all regular polygons,
        \item In $3$ dimensions: the tetrahedron, cube and octahedron,
        \item In $4$ dimensions: the 8-cell (a.k.a.~4-cube), the 16-cell (a.k.a.~4-orthoplex) and the 24-cell;
    \end{itemize}
    there is a transitive soluble group action on the vertex sets of the following remaining regular polytopes:
    \begin{itemize}[topsep = 3pt, itemsep=0pt]
        \item In $3$ dimensions: the icosahedron,
        \item In $4$ dimensions: the 5-cell (a.k.a.~regular 4-simplex),
        \item In $d \ge 5$ dimensions: the regular $d$-simplex, the $d$-cube and the $d$-orthoplex;
    \end{itemize}
    and finally, there is a soluble set $Y$ containing a copy of the vertex set of the dodecahedron.
\end{theorem}

The hardest case of this theorem is the dodecahedron and the following lemma is the key lemma required for settling this case.

\begin{lemma}\label{lem:2-orbit-plus-stuff implies soluble}
    Let $X\subseteq \mathbb{R}^d$ be a transitive set and let $H$ be a group that acts transitively on $X$. Suppose that $G$ is a soluble subgroup of $H$ such that the action of $G$ on $X$ has two orbits $O_1$ and $O_2$, and \[\frac{|H||O_1|}{|G||X|} \le 2.\]
    Then letting $q$ be any prime greater than $|H|/|G|$, there is a soluble set $Y\subseteq \mathbb{R}^{q d}$ containing a copy of $X$.
\end{lemma}
Before we prove Lemma~\ref{lem:2-orbit-plus-stuff implies soluble} we will  deduce the theorem.

% \begin{corollary}\label{lem:2-orbit-plus-stuff implies soluble}
%     Let $X\subseteq \mathbb{R}^d$ be a transitive set. Let $Z \subset X$ and let $G_Z \subset Aut(X)$ be a group of automorphisms of $X$ that fix $Z$. Suppose that $G_Z$ is soluble, the action of $G_Z$ on $X$ has two orbits $Z$ and $X \setminus Z$ and that $|Aut(X)||Z| \le 2|G_Z||X|$.
%     Then there exists a $d'$ and a set $Y\subseteq \mathbb{R}^{d'}$ containing a copy of $X$ such that there is a soluble transitive group action on $Y$.
% \end{corollary}

\begin{proof}[Proof of Theorem~\ref{thm:regular polytopes subsoluble}]

The symmetry groups of regular polytopes are all Coxeter groups, and are listed in Table~\ref{table:coxeter groups}. The group structures as listed in the table are taken from Table 2.1 of~\cite{wilson}. We use $\times$ to denote a direct product, $\rtimes$ to denotes a semi-direct product, $\wr$ to denote a wreath product and $C_2 \cdot$ to denote a double cover, a.k.a.
a non-split extension by $C_2$. For more explanation of these terms see e.g.~\cite{wilson}.

\begin{table}[h]
\begin{tabular}{lccl}\toprule
Group name & Group structure & soluble? & Regular Polytope(s)\\\midrule
$I_2(k)$ & $D_{2k}\cong C_n \rtimes C_2$ &Yes& $k$-gon\\
 $A_n$ & $S_{n+1}$ & Iff $n \le 3$ & $n$-dimensional regular simplex \\
$B_n$ & $C_2 \wr S_n$ & Iff $n \le 4$ & $n$-dimensional cube \\
&&& $n$-dimensional orthoplex\\
$H_3$ & $C_2 \times A_5$ &No& icosahedron \\
&&& dodecahedron\\
$F_4$ & $(C_2\cdot C_2^4)\rtimes(S_3 \times S_3)$ & Yes & $24$-cell\\
$H_4$ & $C_2 \cdot (A_5 \times A_5) \rtimes C_2$ &No & $600$-cell \\
&&& $120$-cell\\
\bottomrule
\end{tabular}
\caption{The symmetry groups of regular polytopes {(group structure from~\cite[Table 2.1]{wilson}})}
\label{table:coxeter groups}
\end{table}

Recall that the permutation group $S_n$ and the alternating group $A_n$ are soluble if and only if $n \le 4$, so any group containing $A_5$ is not soluble. Moreover, for all of the aforementioned products, the product of two soluble groups is soluble. Thus we can easily deduce which of these groups are soluble.

Using Table~\ref{table:coxeter groups}, it is easy to see that all regular polygons, the tetrahedron, the $d$-dimensional cube and orthoplex for $d \le 4$, and the $24$-cell are the only groups with soluble symmetry group.

Now we move onto polytopes with a transitive soluble group action that is not the full automorphism group.

First consider the vertex set $\Delta_{d}$ of the  $d$-dimensional regular simplex embedded in $\mathbb{R}^{d+1}$ with coordinates permutations of $(1,0,0,\ldots,0,0)$. Consider the action of the cyclic group $C_{d+1}$ on $\Delta_{d}$ where $g^i$ acts as a cyclic permutation of the coordinates by $i$ places. That is, if $g$ is the generator of $C_{d+1}$ then for $g^i \in C_{d+1}$ we have
\[g^i:~(x_1,x_2,\ldots,x_{d+1}) ~\mapsto ~(x_{1+i},x_{2+i},\ldots,x_{d+1+i}),\] 
where all subscripts are taken modulo $d+1$. This is clearly transitive on $\Delta_d$.

Next consider the vertex set $Q_d$ of the $d$-dimensional cube in $\mathbb{R}^{d}$  with coordinates $(\pm 1,\ldots, \pm 1)$. Let the soluble group $C_2^d$ act on $Q_d$ by reflections in axis-parallel hyperplanes. In particular, letting $\{-1,1\}$ be the elements of $C_2$ for $(i_1,i_2\ldots i_d) \in C_2^d$ we have \[(i_1,i_2\ldots i_d): ~(x_1,x_2,\ldots,x_d) ~ \mapsto ~ (i_1 x_1, i_2 x_2\ldots,i_d x_d).\]
This is clearly transitive on $Q_d$.

Consider the vertex set $\beta_d$ of the $d$-dimensional orthoplex  in $\mathbb{R}^{d}$ with coordinates permutations of $(\pm 1,0,0,0)$. Let the soluble group $C_2 \times C_d$ act on $\beta_d$ where for ${(i_1,g^{i_2}) \in C_2 \times C_d}$ we have \[(i_1,g^{i_2}):~(x_1,x_2,\ldots,x_d) ~\mapsto ~ (i_1x_{1+i_2},i_1x_{2+i_2},\ldots,i_1x_{d+i_2}),\] 
where all subscripts of $x$ are taken modulo $d$. This is clearly transitive on $\beta_d$.

\begin{figure}[h!]
    \centering
    \includegraphics[width=4.5cm]{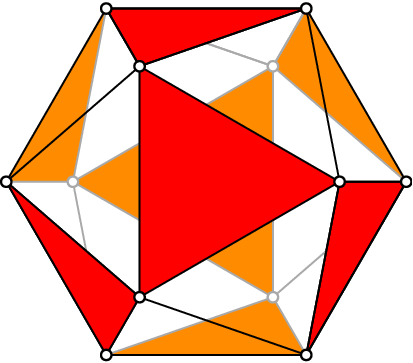}
    \caption{A face-coloured icosahedron with pyritohedral symmetry.}
    \label{fig:icosahedron}
\end{figure}

Now consider the icosahedron. Colour 8 of the faces of the icosahedron red such that the centres of those faces form a cube, as in Figure~\ref{fig:icosahedron}. We will consider the symmetries of the icosahedron with respect to this colouring: that is, those symmetries that map red faces to red faces. This is the set of symmetries of the pyritohedral icosahedron, which is the pyritohedral group $T_h \cong A_4 \times C_2$. This group is clearly soluble. It is straightforward to check that the icosahedron is vertex-transitive under these symmetries. Each vertex is on the boundary of a red face. A series of reflections can send each red face to any other red face, and rotations by $\pi/3$ or $2\pi/3$ about the line through the centre of a red face will send a vertex to the other two vertices of the face. 

    \begin{figure}[h!]
        \centering
        \includegraphics[width=4.5cm]{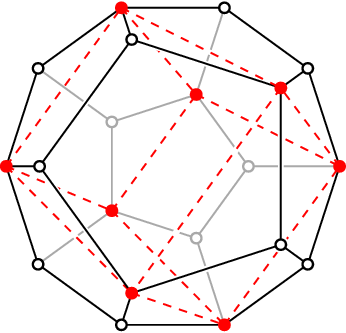}
        \caption{The cube inscribed in the dodecahedron.}
        \label{fig:dodecahedron}
    \end{figure}

Finally, consider the dodecahedron. 
We simply need to check that the conditions of Lemma~\ref{lem:2-orbit-plus-stuff implies soluble} are satisfied. Let $Q \subset X$ be an fixed inscribed cube of the dodecahedron as in Figure~\ref{fig:dodecahedron}. Consider all automorphisms of $X$ that map $Q$ to itself. This is the pyritohedral group $T_h$, a subset of the set of symmetries of the cube. $T_h$ is isomorphic to $A_4 \times C_2$ and is therefore soluble. It is easy to check that $T_h$ has exactly two orbits: $Q$ and $X\setminus Q$.

 We have 
    \[\frac{|Aut(X)||Q|}{|T_h||X|}= \frac{120 \cdot 8}{24 \cdot 20} = 2 \]
    and so the conditions of Lemma~\ref{lem:2-orbit-plus-stuff implies soluble} are satisfied with $H = Aut(X)$, $G = T_h$ and $O_1 = Q$ and $q = 5$.
\end{proof}

Now we will prove the key lemma. 
\begin{proof}[Proof of Lemma~\ref{lem:2-orbit-plus-stuff implies soluble}]
    Fix $y \in O_1$. 
    Note that since $H$ is transitive on $X$, by the orbit-stabilizer theorem we have that $|H|/|X| = |H_y|$, the size of the stabiliser of $y$ in $H$. Since the orbit of $y$ under $G$ is $O_1$, we also have that $|G|/|O_1| = |G_y|$, the size of the stabiliser of $y$ in $G$. Note that $G_y$ is a subgroup of $H_y$ and so $|H_y|/|G_y|  = |H||O_1| / |G||X|$ must be a positive integer. Therefore, by the condition in the statement of the lemma $|H||O_1| / |G||X|$ must be equal to $1$ or $2$.

    Let $s \coloneqq |H|/|G|$ and let $p$ be a prime greater than $s$.
    Define $Y \subseteq X^p$ to be all 
    $(x_1,\ldots,x_p) \in X^p$ such that exactly $|H||O_1| / |G||X|$ of the entries $x_i$ lie in $O_1$.

    First we will show that $Y$ contains a scale copy of $X$: then we can take an appropriate rescaling of $Y$ for our final set as stated in the lemma. Fix any $z \in O_2$. Fix a collection $f_1,f_2,\ldots,f_s$ of elements of $H$, one for each right coset of $G$. 
    For $x \in X$, define  
    \[v(x)  \coloneqq 
    (f_{1}(x), \ldots ,f_{s}(x), \underbrace{z, \ldots, z}_{p - s}~) 
    \]
    and define 
    \[
    X' \coloneq \{v(x) : x \in X\} = \{(f_{1}(x), \ldots ,f_{s}(x), \underbrace{z, \ldots, z}_{p - s}~) : x \in X\} .
    \]
    We have $v(x) \in X^p$ for all $x$ and so $X' \subseteq X^p$. Clearly $X'$ is a copy of $X$ scaled by $\sqrt{s}$. We must show that $X' \subseteq Y$.
%\begin{proof}[Proof of Claim]
    Take any $x \in X$. Let $C(x)$ be the number of entries of $v(x)$ that lie in $O_1$. By definition, $C(x)$ is equal to the number of $i \in [s]$ such that $f_i(x) \in O_1$. 
    Note that all elements of $H$ can be expressed as $gf_i(x)$ for some $1 \le i \le s$ and some $g \in G$, and moreover, $gf_i(x) \in O_1$ if and only if $f_i(x)$ is. 
    Therefore, $C(x)|G|$ is equal to the number of $h \in H$ such that $h(x) \in O_1$. However, as $H$ acts transitively on $X$, the number of $h \in H$ such that $h(x) \in O_1$ is exactly $|H||O_1|/|X|$. 
    Combining these, we get $C(x) = |H||O_1|/|G||X|$. In particular, $v(x) \in Y$ for all $x \in X$. 
 
    % First, suppose that $x \in O_1$. Since $g_i(x) \in O_1$ for all $g \in G$ and $f(O_1) \subseteq O_2$, we see that $fg_{1}(x) \in O_2$ for $1 \le i \le |G|$. Thus exactly $s|O_1|$ of the entries of $v(x)$ lie in $O_1$ and in particular, $v(x) \in Y$.

    % Now suppose that $x \in O_2$.  We wish to determine the number of elements of $G$ such that $fg(x)$ is in $O_1$. The number of $y \in O_2$ such that $f(y) \in O_1$ is exactly $|O_1|$, since $f(O_2) \supseteq O_1$. 
    % As $G$ is transitive on $O_2$, for each fixed $y \in O_2$ the number of $g \in G$ such that $g(x) = y$ is $|G|/|O_2| = s$. Therefore, the number of $1\le i \le |G|$ such that $fg_i(x) \in O_1$ is exactly $s|O_1|$. In particular, $v(x) \in Y$. 
%\end{proof}

    Now define $G'$ to be the wreath product $G \wr AGL(1,p)$; that is, $G' =  G^p \rtimes AGL(1,p)$ where \[(g_1,g_2,\ldots,g_p;\phi)\cdot (h_1,h_2,\ldots,h_p;\psi) = (g_1h_{\phi^{-1}(1)},g_2h_{\phi^{-1}(2)},\ldots,g_ph_{\phi^{-1}(p)};\phi\psi).\] As $G$ and  $AGL(1,p)$ are both soluble, so is $G'$. 
    
    We define an action of $G'$ on $Y$.
    For a group element $(g_1,g_2,\ldots,g_p;\phi) \in G'$ and a point $(x_1,x_2,\ldots,x_p) \in Y$, define \[(g_1,g_2,\ldots,g_p;\phi):~(x_1,x_2,\ldots,x_p) \mapsto (g_{1}(x_{\phi^{-1}(1)}), g_{2}(x_{\phi^{-1}(2)}),\ldots, g_{p}(x_{\phi^{-1}(p)}).\] Intuitively, the entries are permuted according to $\phi$, and then group element $g_i$ is applied to the resulting $i$th entry $x_{\phi(i)}$. Since the action of $g_i$ preserves membership of $O_1$ or $O_2$ for all $i$, we can see that $h(Y) \subseteq Y$. 
    Moreover,
    \begin{align*}
    (g_1,\ldots,g_p;\phi)(h_1,\ldots,h_p;\psi)(x_1,\ldots,x_p) &= 
     (g_1,\ldots,g_p;\phi)(h_{1}(x_{\psi^{-1}(1)}),\ldots, h_{p}(x_{\psi^{-1}(p)})
    \\
    &= (g_1h_{\phi^{-1}(1)}(x_{\psi^{-1}\phi^{-1}(1)}),\ldots, g_ph_{\phi^{-1}(p)}(x_{\psi^{-1}\phi^{-1}(p)})) \\
    &= (g_1h_{\phi^{-1}(1)},\ldots,g_ph_{\phi^{-1}(p)};\phi\psi)(x_1,\ldots,x_p) 
    \end{align*}
    and so this action is well-defined.

 %   \begin{claim}
    All that remains is to show that the action of $G'$ on $Y$ is transitive. 
    % \end{claim}
    % \begin{proof}
    Fix a point $y \in O_1$ and $z \in O_2$. 
    
    First suppose that $|H||O_1|/|G||X| = 1$. We will show that $(y,\underbrace{z,\ldots,z}_{p-1}) \in Y$ can be transformed under the action of $G'$ to any $(x_1,\ldots,x_p) \in Y$. Let $x_t$ be the only entry of $(x_1,\ldots,x_p)$ that lies in $O_1$. Let $\phi : \mathbb{Z}_p \rightarrow \mathbb{Z}_p$ send $a \mapsto ta$. Then let $g_t \in G$ be chosen so that $g_t(y) = x_t$ and for $i \ne t$, let $g_i \in G$ be such that $g_i(z) = x_i$. This is possible since $G$ acts transitively on $O_1$ and $O_2$. Then \[(g_1,g_2,\ldots,g_p;\phi)(y,\underbrace{z,\ldots,z}_{p-1}) %= (g_1(z),g_2(z),\ldots, g_{t-1}(z),g_t(y),g_{t+1}(z),\ldots,g_p(z)) 
    = (x_1,\ldots,x_p).\]
    % and taking $h' =  (\phi;g^{-1}_{\phi(1)},g^{-1}_{\phi(2)},\ldots,g^{-1}_{\phi(p)})$ gives
    % \[h'((x_1,\ldots,x_p)) = (g_{t}^{-1}(x_{t}),g_{2t}^{-1}(x_{2t}),\ldots,g^{-1}_{pt}(x_{pt})) = (y,\underbrace{z,\ldots,z}_{p-1}).\] 

    Now suppose that $|H||O_1|/|G||X| = 2$. We will similarly show that $(y,y,\underbrace{z,\ldots,z}_{p-2}) \in Y$ can be transformed under the action of $G'$ to any $(x_1,\ldots,x_p) \in Y$ and vice versa. Let $x_t$, $x_s$ be the two entries of $(x_1,\ldots,x_p)$ that lie in $O_1$. 
    Let $\phi : \mathbb{Z}_p \rightarrow \mathbb{Z}_p$ send $a \mapsto (s-t)a + (2t - s)$ so that $\phi(1) = t$ and $\phi(2) = s$. Then let $g_t \in G$ be such that $g_t(y) = x_t$, $g_s$ be such that $g_s(y) = x_s$, and for $i \ne t,s$, let $g_i \in G$ be such that $g_i(z) = x_i$. 
    Then \[(g_1,g_2,\ldots,g_p;\phi)(y,y,\underbrace{z,\ldots,z}_{p-2}) %= (g_1(z),g_2(z),\ldots, g_t(y),\ldots,g_s(y),\ldots,g_p(z)) 
    = (x_1,\ldots,x_p)\]
    % and taking $h' =  (\phi; g^{-1}_{\phi(1)},g^{-1}_{\phi(2)},\ldots,g^{-1}_{\phi(p)})$ gives
    % \[h'((x_1,\ldots,x_p)) = (g_{t}^{-1}(x_{t}),g_{s}^{-1}(x_{s}),g_{\phi(3)}^{-1}(x_{\phi(3)}),\ldots,g^{-1}_{\phi(p)}(x_{\phi(p)})) = (y,y,\underbrace{z,\ldots,z}_{p-2}),\] 
    as required.
%    \end{proof}
\end{proof}

\section{Isosceles trapezia are subsoluble}\label{sec:trapezia}
In this section we prove that isosceles trapezia are subsoluble. This proof is inspired by an observation of Leader, Russell and Walters~\cite{blocksets} that an arbitrary triangle lies in the vertex set of some twisted prism with a $k$-gon base. Since such a twisted prism is acted on transitively by the soluble group $D_{2k}\times C_2$ and  all triangles are therefore subsoluble. This same proof was also explained in greater detail in~\cite{JOHNSON}.
We show that a similar property holds for isosceles trapezia.
\begin{theorem}
    All isosceles trapezia are subsoluble.
\end{theorem}
\begin{proof}
    Let $ABCD$ be an isosceles trapezium such that the edges $AD$ and $BC$ are parallel, and the edges $AB$ and $CD$ are the same length. Note that all isosceles trapezia are cyclic quadrilaterals. Let  $\mathcal{C}$ be the circle circumscribing $ABCD$ with centre $O$.
    \begin{claim}
        Suppose angle $\angle AOB$ is equal to $2\pi / k$ for some integer $k$. Then there exists a finite set $Y \subseteq \mathcal{C}$ containing the vertices $A,B,C,D$ such that there is a soluble transitive group action on $Y$.
    \end{claim}
    \begin{proof}[Proof of Claim]
        Since $\angle AOB$ is equal to $2\pi / k$ for some integer $k$, there exists a regular $k$-gon $X \subseteq \mathcal{C}$ containing $A$ and $B$. Reflect $X$ in the perpendicular bisector of $AD$ to get a regular $k$-gon $X'$ containing $C$ and $D$.  Note that the perpendicular bisector of $AD$ passes through $O$ and so $X' \subseteq \mathcal{C}$. Let $Y = X \cup X'$. Clearly by construction $Y$ contains $A,B,C,D$. 

        Let $G$ be the group generated by a rotation $r$ about $O$ by $2\pi /k$ and the reflection $s$ in the perpendicular bisector of $AD$. Clearly $G$ acts transitively on $Y$. As $G$ is isomorphic to the dihedral group $D_{2k}$, it is also soluble.
    \end{proof}
    Now suppose that angle $\angle AOB$ is not equal to $2\pi / k$ for any integer $k$. Call the line through $B$ perpendicular to $AD$ the \emph{altitude} from $B$, and similarly for $C$.    
    Consider the isosceles trapezia obtained by moving $B$ and $C$ towards $AD$ along their respective altitudes, calling the resulting vertices $B',C'$. At the limit, as $B,C$ get arbitrarily close to $AD$, the angle $\angle AO'B'$ gets arbitrarily small. Therefore there must be some choice of $B',C'$ where $\angle AO'B'$ is $2\pi /k$ for some integer $k$. Fix this choice of $B'$ and $C'$.

    By the claim, there is a soluble set $Y$ containing $AB'C'D$ that is contained in the circle $\mathcal{C}'$ containing $AB'C'D$. Working now in $3$ dimensions, consider the set $Z = Y\times \{-x,x\}$, where $x$ satisfies $(2x)^2 + |AB'|^2 = |AB|^2$. Note that since $Y$ contains $AB'C'D$ then the set $Z$ contains an isometric copy of $ABCD$. Moreover, letting $G$ be the soluble group that acts transitively on $Y$, then extending $G$ by the reflection in the $xy$-axis gives a soluble group that acts transitively on $Z$.
\end{proof}
\section{Open questions}\label{sec:openqs}

By \kriz's arguments, we know that if a set $X$ is a subset of a transitive set $Y$ with a soluble group action then $X$ is Ramsey. We have observed that in nearly all known examples of Ramsey sets, the converse is true. We ask if the converse holds in general.
\begin{question}\label{q:all subsoluble}
    Let $X \subseteq \mathbb{R}^d$ be Ramsey. Must there exist a soluble set $Y$ containing a copy of $X$? 
\end{question}

\kriz ~in fact proved the stronger result that if $X$ is a finite set with some soluble group acting on $X$ with at most two orbits, then $X$ is Ramsey. 
The $120$-cell and $600$-cell were shown to be Ramsey by repeated application of this stronger result~\cite{cantwell}. In particular, answering the following question would resolve Question~\ref{q:all subsoluble} for all sets currently known to be Ramsey.
\begin{question}\label{q:2-orbit imply soluble}
    Let $X$ be a transitive set and suppose that some soluble group acts on $X$ with two orbits. Does there exist a soluble set $Y$ containing $X$?
\end{question}

We have shown in Theorem~\ref{thm:regular polytopes subsoluble} that nearly all regular polytopes are subsoluble. It would be particularly interesting to determine whether the two remaining regular polytopes are subsoluble. 
\begin{question}\label{q:regular polytopes} Does the vertex set of a regular polytope always embed within some soluble set $Y$? In particular, does this hold for the $120$-cell and the $600$-cell? 
\end{question}
Since the vertex set of the $600$-cell is contained within the vertex set of the $120$-cell, it would be enough to answer Question~\ref{q:regular polytopes} for the $120$-cell. A first attempt (for either case) is to try to duplicate the proof that the dodecahedron is subsoluble: that is, to fix a particular inscribed 4-dimensional regular polytope $Q$ and consider the group $G$ of automorphisms of $X$ that map $Q$ to itself. However, in all cases one of the following conditions necessary to apply Lemma~\ref{lem:2-orbit-plus-stuff implies soluble} fails to hold: either $G$ is not soluble, the action of $G$ on $X$ has more than two orbits, or $\frac{|Aut(X)||Q|}{|G||X|} > 2$. Thus a new idea is required to answer Question~\ref{q:regular polytopes}.

\bigskip

\bigskip

We showed in Theorem~\ref{thm:blocksets subsoluble} that sets containing all permutations of certain sets of coordinates are subsoluble. We wonder whether it is possible to extend this result further. 
\begin{question}
    Let $\alpha_1,\alpha_2,\ldots,\alpha_s \in \mathbb{R}$ (not necessarily distinct), and let $i_1,i_2,\ldots,i_s \ge 0$. 
      Let $X(i_1,i_2,\ldots,i_s) \subset \mathbb{R}^{i_1 + i_2+\ldots+i_s}$ be the set containing permutations of the coordinates \[(\underbrace{\alpha_1,\ldots,\alpha_1}_{i_1},\underbrace{\alpha_2,\ldots,\alpha_2}_{i_2},\ldots,\underbrace{\alpha_s,\ldots,\alpha_s}_{i_s}).\] 
    For which values of $i_1,\ldots,i_s$ is the set $X(i_1,\ldots,i_s)$ subsoluble for any choice of $\alpha_1,\ldots,\alpha_s$?
    In particular, is $X(1,1,1,1,1)$ subsoluble for all  $\alpha_1,\alpha_2,\ldots,\alpha_5$?
\end{question}
These sets arise from block sets in a natural way, and a positive answer to the Block Set Conjecture \cite{blocksets} would imply that $X(i_1,i_2,\ldots,i_s)$ is always Ramsey. However, neither result seems to directly imply the other - we do not know whether being Ramsey implies subsoluble (this is Question~\ref{q:all subsoluble}), and although we have proved that $X(i,j,1,1)$ is always Ramsey we have been unable to deduce anything about the corresponding block sets.

Our approach in proving Theorem~\ref{thm:blocksets subsoluble} cannot be directly extended to any other patterns without a new idea. We use $AGL(1,p)$ in a crucial way within the proof as a group that is both soluble and is $2$-transitive on $\mathbb{Z}_p$: that is, any pair of elements of $\mathbb{Z}_p$ can be sent to any other under the action of  $AGL(1,p)$. To extend the proof to new $X(i_1,i_2,\ldots,i_s)$ we would need to find a soluble group acting on $\{1,2,\ldots,n\}$ for some $n \ge 5$ whose action is $3$-transitive or higher. However, one can use the Classification of Finite Simple Groups to show that such groups do not exist. A list of all multiply transitive groups can be found in Tables 7.3 and 7.4 of~\cite{Cameron}, and it is a matter of routine to check that among these there is no soluble group which acts $3$-transitively on a set of size $\ge 5$.

\bigskip 

 Clearly, a positive answer to Question~\ref{q:all subsoluble} would immediately imply one direction of Conjecture~\ref{conj:rival} that every Ramsey set is a subset of a transitive set. To determine whether a negative answer
to Question~\ref{q:all subsoluble}  would contradict the conjecture, we would need an answer to
the following question.
\begin{question}
Is every transitive set subsoluble?
\end{question}

See Figure~\ref{fig:dependencies} for a diagram of the known and conjectured implications between the properties of finite sets in Euclidean space discussed above.
\begin{figure}[h!]
    \centering
    \includegraphics[scale=.9]{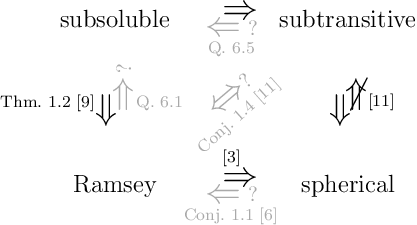}
    \caption{Implications between properties of finite sets in Euclidean space}
    \label{fig:dependencies}
\end{figure}

\section{Acknowledgements}

Thank you to Imre Leader for helpful feedback on the first draft of this paper. Thank you also to Arsenii Sagdeev for drawing my attention to the paper of Karamanlis on simplices~\cite{Karamanlis}.

This research was supported by the European Research Council (ERC) under the European Union Horizon 2020 research and innovation programme (grant agreement No. 947978).

\bibliographystyle{abbrv}
\bibliography{ref}

\end{document}